\documentclass [twoside,reqno, 12pt] {amsart}

\usepackage{tikz}

\usepackage{amsfonts}
\usepackage{amssymb}
\usepackage{a4}
\usepackage{color}

\newtheorem{thm}{Theorem}[section]
\newtheorem{cor}[thm]{Corollary}

\newtheorem{prop}[thm]{Proposition}

\theoremstyle{definition}

\numberwithin{equation}{section}

\renewcommand{\Re}{\hbox{Re}\,}
\renewcommand{\Im}{\hbox{Im}\,}

\newcommand{\C}{\mathbb{C}}

\newcommand{\N}{\mathbb{N}}

\newcommand{\R}{\mathbb{R}}

\newcommand{\supp}{\operatorname{supp}}

\def\tilde{\widetilde}
\def \bfo {\begin {eqnarray*} }
\def \efo {\end {eqnarray*} }
\def \ba {\begin {eqnarray*} }
\def \ea {\end {eqnarray*} }
\def \beq {\begin {eqnarray}}
\def \eeq {\end {eqnarray}}
\def \supp {\hbox{supp }}

\def \p {\partial}

\def\tilde{\widetilde}
\def \bfo {\begin {eqnarray*} }
\def \efo {\end {eqnarray*} }
\def \ba {\begin {eqnarray*} }
\def \ea {\end {eqnarray*} }
\def \beq {\begin {eqnarray}}
\def \eeq {\end {eqnarray}}
\def \supp {\hbox{supp }}

\def \p {\partial}

\begin{document}

 \title[Bounds on eigenfunctions]{Bounds on eigenfunctions of semiclassical operators with double characteristics}
\author[Krupchyk]{Katya Krupchyk}

\address
        {K. Krupchyk, Department of Mathematics\\
University of California, Irvine\\ 
CA 92697-3875, USA }

\email{katya.krupchyk@uci.edu}

\author[Uhlmann]{Gunther Uhlmann}

\address
       {G. Uhlmann, Department of Mathematics\\
       University of Washington\\
       Seattle, WA  98195-4350\\
       USA}
\email{gunther@math.washington.edu}

\maketitle

\begin{abstract}
We obtain sharp uniform bounds on the low lying eigenfunctions for a class of semiclassical pseudodifferential operators with double characteristics and complex valued symbols, under the assumption that the quadratic approximations along the double characteristics are elliptic.

\end{abstract}

\section{Introduction and statement of results}

This paper is motivated by the study of the semiclassical Schr\"odinger operator 
\[
P=-h^2\Delta+V(x)\quad \text{on}\quad  \R^n,
\]
where $V\in C^\infty(\R^n;\R)$ has non-degenerate potential wells. Such operators are of great significance in quantum mechanics as well as in geometry, see \cite{Simon_book_87} and \cite{Helffer_book_1988}.

When the potential $V$ is such that $V\ge 0$ and $\liminf_{|x|\to \infty}V(x)>0$, taking the Friedrichs extension of $P$,  we obtain a nonnegative self-adjoint operator on $L^2(\R^n)$ with discrete spectrum in an interval of the form $[0,\delta]$ with $\delta>0$ small but fixed, see \cite[p. 37]{Dimassi_Sjostrand}.  The eigenvalues of $P$ in an interval of the form $[0,Ch]$, when $h\to 0$, known as the low lying eigenvalues, are of primary importance in quantum mechanics. 

To recall some precise results concerning low lying eigenvalues and corresponding eigenfunctions of $P$, let  us assume for simplicity that the potential $V$ is bounded with all derivatives and that $V$ has a unique minimum which is non-degenerate and achieved at $x=0$  so that $V^{-1}(0)=\{0\}$, $V'(0)=0$, and $V''(0)>0$.  Taylor expanding the symbol $p(x,\xi)=\xi^2+V(x)$ of $P$ at $(0,0)$, we get
\begin{equation}
\label{eq_introduction_app_quad}
p(x,\xi)=q(x,\xi)+\mathcal{O}(x^3),\quad q(x,\xi)=\xi^2+\frac{1}{2} V''(0)x\cdot x.
\end{equation}
Thanks to the works  \cite{Simon_1983} and \cite{Helffer_Sjostrand_1984}, it is known that 
the low lying eigenvalues $\lambda(h)$ of $P$ enjoy complete asymptotic expansions of the form
\begin{equation}
\label{eq_introduction_asymp_Schr}
 \lambda(h)\sim h(E_0+h^{1/2}E_1+hE_2+\dots), \quad h\to 0,
\end{equation}
where $E_0$ is an eigenvalue of the quadratic operator $q(x,D_x)=-\Delta +\frac{1}{2} V''(0)x\cdot x$, and $E_j\in \R$, $j=0,1,\dots$.   
Turning the attention to the corresponding low lying eigenfunctions of $P$,  from the works \cite{Helffer_Sjostrand_1984},  \cite{Helffer_Sjostrand_1985}, and \cite{Dimassi_Sjostrand},  we know that they can be well approximated by suitable WKB expressions. Specifically, when $\lambda(h)\in [0,Ch]$ is a simple eigenvalue of $P$, the corresponding $L^2$--normalized eigenfunction $u(x;h)$ has the form
\begin{equation}
\label{eq_introduction_uniform_1}
u(x;h)=h^{-\frac{n}{4}}e^{-\frac{\varphi(x)}{h}}(a(x;h)+\mathcal{O}(h^\infty)),
\end{equation}
in a small neighborhood of $x=0$. Here $\varphi\in C^\infty(\text{neigh}(0,\R^n);\R)$ is such that $\varphi(0)=0$, $\varphi'(0)=0$, $\varphi''(0)>0$, and $a(x;h)$ is smooth in $x$ with an asymptotic expansion in powers of $h$.  Away from a small neighborhood of $0$, the eigenfunction $u(x;h)$ is exponentially decaying, see \cite[Chapter 6]{Dimassi_Sjostrand}, and in particular it follows that $u\in L^\infty(\R^n)$ and 
\begin{equation}
\label{eq_introduction_uniform}
\|u\|_{L^\infty}\le \mathcal{O}(1) h^{-\frac{n}{4}}.
\end{equation}

Now in many problems of mathematical physics, ranging from fluid dynamics and theory of superconductivity to kinetic theory, one encounters more general semiclassical operators, including some non-self-adjoint ones, such as Schr\"odinger operators with complex potentials as well as operators of Kramers-Fokker-Planck type, see  \cite{Almog_Henry}, \cite{Davies_non_self-adj}, \cite{Helffer_Nier}. 
 A basic feature of such operators is that similarly to \eqref{eq_introduction_app_quad} they can locally be modeled 
by quadratic differential operators, sometimes satisfying suitable ellipticity conditions.

In this paper we are interested in the study of  low lying eigenfunctions for such more general  semiclassical pseudodifferential operators, including non-self-adjoint ones. Specifically,   we shall be concerned with operators of the form
\begin{equation}
\label{eq_int_new_5}
P=\text{Op}_h^w(p_0)+h\text{Op}_h^w(p_1) \quad \text{on}\quad \R^n,\quad n\ge 2,
\end{equation}
where $\text{Op}_h^w(p)$ is the semiclassical Weyl quantization of a symbol $p=p(x,\xi;h)$, 
\begin{equation}
\label{eq_1_weyl}
(\text{Op}_h^w(p)u)(x)=\frac{1}{(2\pi h)^n}\int_{\R^n}\int_{\R^n} e^{\frac{i}{h}(x-y)\cdot\xi}p\bigg(\frac{x+y}{2},\xi;h\bigg)u(y)dyd\xi.
\end{equation}
Here $0<h\le 1$ is the semiclassical parameter. 

Let us state our assumptions on the symbols $p_0$ and $p_1$ in \eqref{eq_int_new_5}. First we assume that $p_0\in C^\infty(\R^{2n};\C)$, independent of $h$,  is such that  
\begin{equation}
\label{eq_assumption_5}
\p^{\alpha}p_0\in L^\infty(\R^{2n}), \quad \alpha\in \N^{2n}, \quad |\alpha|\ge 2. 
\end{equation}
We assume that 
\begin{equation}
\label{eq_assumption_1}
\Re p_0(X)\ge 0,\quad  X=(x,\xi)\in \R^{2n},
\end{equation}
and we also make the assumption of  ellipticity at infinity for $\Re p_0$ in the sense that for some $C>1$, 
\begin{equation}
\label{eq_assumption_4}
\Re p_0(X)\ge \frac{\langle X\rangle^2}{C}, \quad  |X|\ge C.
\end{equation} 
Here $\langle X\rangle=\sqrt{1+|X|^2}$.  Furthermore, let us assume that 
\begin{equation}
\label{eq_assumption_2_1}
(\Re p_0)^{-1}(0)=\{0\}. 
\end{equation}
Notice that \eqref{eq_assumption_2_1} and \eqref{eq_assumption_1} imply that 
\[
\nabla \Re p_0(0)=0.
\]
Next we assume that 
\[
\Im p_0(0)=\nabla \Im p_0(0)=0,
\]
so that $X=0$ is a doubly characteristic point for the full complex valued symbol $p_0$.  
By Taylor's expansion, we write 
\begin{equation}
\label{eq_assump_quadratic_form}
p_0(X)=q_0(X)+\mathcal{O}(|X|^3), \quad \text{as}\quad |X|\to 0,
\end{equation}
where 
\[
q_0(X)=\frac{1}{2}p_0''(0)X\cdot X,
\]
and  $p''_0$ is the Hessian of $p_0$.  In view of  \eqref{eq_assumption_1}, we know that $\Re q_0(X)\ge 0$, $X\in \R^{2n}$. Our final assumption on $p_0$ is that the quadratic form $\Re q_0$ is positive definite, i.e. 
\begin{equation}
\label{eq_assumption_2_1_1}
\Re q_0(X)>0, \quad 0\ne X\in \R^{2n}. 
\end{equation}

Regarding the symbol $p_1$ in \eqref{eq_int_new_5}, we assume that $p_1(X; h)\in C^\infty(\R^{2n};\C)$ and similarly to \eqref{eq_assumption_5}, we also assume that  
\begin{equation}
\label{eq_assumption_p_1_der}
\p^\alpha p_1\in L^\infty(\R^{2n}) , \quad \alpha\in \N^{2n},\quad |\alpha|\ge 2,
\end{equation}
uniformly in $h\in (0,1]$. 

Let us mention that the study of operators with double characteristics has long played a prominent role in the theory of linear PDE, and we refer to \cite{Boutet_1974}, \cite{Sjostrand_1974}, \cite{Hormander_1975},  \cite{Uhmann_1977} for some of the fundamental results in this area.  
 
\textbf{Example.} As a simple yet significant example of an operator for which all the assumptions above are satisfied, let us consider  a Schr\"odinger operator with a complex potential,   
\[
P=-h^2\Delta+V(x)+iW(x)\quad\text{on}\quad \R^n, \quad n\ge 2.
\]
Here $V,W\in C^\infty(\R^n;\R)$ are such that  $\p^\alpha V, \p^\alpha W\in L^\infty(\R^n)$  for $|\alpha|\ge 2$.  We assume that $V(x)\ge 0$ for  $x\in \R^n$ and $V(x)\ge |x|^2/C$ for $|x|\ge C$. Furthermore, assume that 
$V^{-1}(0)=\{0\}$,  $V''(0)>0$, and $W(0)=\nabla W(0)=0$.  

Coming back to the operator $P$ in \eqref{eq_int_new_5}, we shall view it  as a closed densely defined operator on $L^2(\R^n)$, equipped with the domain 
\[
\mathcal{D}(P)=\{ u\in L^2(\R^n): (-h^2\Delta +|x|^2)u\in L^2(\R^n)\}.
\]
We notice that the inclusion map $\mathcal{D}(P)\hookrightarrow L^2(\R^n)$ is compact, and hence, the spectrum of $P$ is discrete.  The low lying eigenfunctions considered in this work correspond to the eigenvalues of $P$ in 
an open disc $D(0,Ch)$ of radius $Ch$, centered at the origin.  

Thanks to the works  \cite{Sjostrand_1974}, \cite{Boutet_1974},  \cite{Herau_Sjostrand_Stolk} and \cite{Hitrik_Pravda-Starov_3},  we know that the eigenvalues of $P$ in the disc $D(0,Ch)$ enjoy complete asymptotic expansions when the subprincipal  symbol $p_1$ in \eqref{eq_int_new_5} is such that 
\[
p_1(x,\xi;h)\sim\sum_{j=0}^\infty h^{j}p_{1,j}(x,\xi).
\]
Specifically, similarly to \eqref{eq_introduction_asymp_Schr},  for any $C>0$,  there exists $h_0>0$ such that for all $0<h\le h_0$, the eigenvalues $\lambda_k(h)$ of $P$ in $D(0,Ch)$ are given by
\[
\lambda_k(h)\sim h(\mu_k+p_{1,0}(0)+h^{1/N_k}\mu_{k,1}+h^{2/N_k}\mu_{k,2}+\dots),
\]
where $\mu_k$ are the eigenvalues of  $\text{Op}_1^w(q_0)$ in $D(0,C)$, repeated with their algebraic multiplicity $N_k\in \N$. The eigenvalues $\mu_k$ can be computed explicitly, see \cite{Sjostrand_1974}, \cite{Boutet_1974}.  

Turning the attention to the low lying eigenfunctions of $P$ in \eqref{eq_int_new_5}, let us remark that obtaining WKB approximations for the eigenfunctions similar to \eqref{eq_introduction_uniform_1} seems to be out of reach in general. Nevertheless, one can still hope for precise 
bounds of the form \eqref{eq_introduction_uniform} for  the low lying eigenfunctions of $P$. It turns out that this hope is justified, as the following theorem, which is the main result of this work, shows. To state this result, we shall equivalently be concerned with an equation of the form $Pu=0$.    
\begin{thm}
\label{thm_main}
Assume that 
$u\in L^2(\R^n)$, $\|u\|_{L^2}=1$, is such that 
\[
(\emph{\text{Op}}_h^w(p_0)+h\emph{\text{Op}}_h^w(p_1))u=0\quad \text{on}\quad \R^n, \quad n\ge 2.
\]
There exists $h_0>0$ such that for all $h\in (0,h_0]$, we have $u\in L^\infty(\R^n)$
and
\begin{equation}
\label{eq_int_est_main_2}
\|u\|_{L^\infty}\le \mathcal{O}(1) h^{-\frac{n}{4}}.
\end{equation}
Hence, by interpolation, 
\begin{equation}
\label{eq_int_est_main_3}
\|u\|_{L^p}\le \mathcal{O}(1) h^{\frac{n}{2p}-\frac{n}{4}}, \quad  \quad 2\le p\le \infty. 
\end{equation}
\end{thm}

The proof of Theorem \ref{thm_main} gives also the following stronger result. 
\begin{cor}
\label{cor_main}
 Assume that 
$u\in L^2(\R^n)$, $\|u\|_{L^2}=1$, is such that 
\[
(\emph{\text{Op}}_h^w(p_0)+h\emph{\text{Op}}_h^w(p_1))u=0\quad \text{on}\quad \R^n, \quad n\ge 2.
\]
 Then for any $K\in \N$, there exists $h_0>0$ such that for all $h\in (0,h_0]$, we have $u\in C^\infty(\R^n)$ and
\begin{equation}
\label{eq_est_cor}
\bigg\| \bigg(\frac{x}{h^{1/2}}\bigg)^\alpha (h^{1/2}\p_x)^\beta u(x) \bigg\|_{L^\infty(\R^n)}\le \mathcal{O}_K(h^{-n/4}),
\end{equation}
for all $\alpha,\beta\in\N^n$, $|\alpha+\beta|\le K$.
\end{cor}

 The following example shows that the estimates \eqref{eq_int_est_main_2}, \eqref{eq_int_est_main_3} and \eqref{eq_est_cor} are sharp within our class of operators. 

\textbf{Example.} Consider the quantum harmonic oscillator, 
\[
P=-h^2\Delta+|x|^2, \quad x\in \R^n, \quad n\ge 2. 
\] 
The operator $P$, equipped with the domain,
\begin{align*}
\mathcal{D}(P)=\{u\in L^2(\R^n): x^\alpha\p_x^\beta u\in L^2(\R^n), |\alpha+\beta|\le 2\},
\end{align*} 
 is self-adjoint  with  discrete spectrum given by 
 \[
 \lambda_\alpha(h):=(2|\alpha|+n)h,\quad \alpha\in \N^n. 
 \]
The corresponding $L^2$ normalized eigenfunctions are of the form
\[
u_\alpha(h)(x)=h^{-\frac{n}{4}}p_\alpha(x/h^{1/2}) e^{-\frac{|x|^2}{2h}},
\] 
 where $p_\alpha$ are the Hermite polynomials of degree $|\alpha|$, see  \cite[Section 6.1]{Zworski_book}.
A direct computation shows that 
\[
\|u_\alpha(h)\|_{L^p}=C_\alpha h^{\frac{n}{2p}-\frac{n}{4}}, \quad  \quad 2\le p\le \infty,
\]
where 
\[
C_\alpha=\bigg(\int_{\R^n} |p_\alpha(x)|^pe^{-\frac{|x|^2p}{2}}dx\bigg)^{1/p}.
\]
It follows that the bounds \eqref{eq_int_est_main_2} and \eqref{eq_int_est_main_3} are saturated by the ground state eigenfunctions $u_\alpha(h)$, corresponding to $\lambda_\alpha(h)\le \mathcal{O}(h)$. The sharpness of \eqref{eq_est_cor}, for any $K\in \N$, is verified similarly.

\textbf{Remark.} Let us emphasize that the uniform estimate \eqref{eq_int_est_main_2} is valid in the case when the principal symbol $p_0$ is complex valued. In the case when $p_0$ is real-valued, the general results of the works \cite[Theorem 6]{Koch_Tataru_Zworski_2007} and \cite{Smith_Zworski_2013}, valid also for higher energy quasimodes, are available, and specifying these results to the low lying eigenfunctions of $P$, we get the following bound
\[
\|u\|_{L^\infty}\le \mathcal{O}(1)h^{-\frac{(n-1)}{2}},  
\]
which can be compared with  \eqref{eq_int_est_main_2}.  

\textbf{Remark.} Let us mention that Theorem \ref{thm_main} and Corollary \ref{cor_main} can be proved relying on the analysis developed in \cite{Tataru_Fefferman-Phong}. Our approach here is different and is based on direct techniques of semiclassical analysis, see \cite{Dimassi_Sjostrand}, \cite{Zworski_book}.

Let us now describe the main idea of the proof of Theorem \ref{thm_main} and the plan of the paper. Heuristically, we expect solutions $u$ of the equation  
\[
(\text{Op}_h^w(p_0)+h\text{Op}_h^w(p_1))u=0
\]
  to be concentrated to the region $p_0(X)+hp_1(X;h)=0$.  It follows from our assumptions that in this region, 
\[
|X|^2/C\le |p_0(X)|=h|p_1(X;h)|\le Ch(1+|X|^2),
\]
and therefore, for $h$ small enough, we conclude that $u$ should be concentrated to the region
\[
|X|^2\le Ch. 
\]
Hence, one wishes to microlocalize $u$ by means of $h$-pseudodifferential operators of the form 
\begin{equation}
\label{eq_int_scale_h_1_2}
\text{Op}_h^w(\chi(X/h^{1/2})), \quad \chi\in C^\infty_0(\R^n). 
\end{equation}
Since the symbols $\chi(X/h^{1/2})$ are only regular on the scale $h^{1/2}$, we know from \cite[Theorem 4.17]{Zworski_book} that the operators \eqref{eq_int_scale_h_1_2}  belong to a calculus having no asymptotic expansion in powers of $h$. A suitable exotic $h^{1/2}$ calculus, involving two small parameters $0<h\le \tilde h\ll 1$,  was developed in \cite{Sjostrand_Zworski_2007}, see also \cite{Datchev_Dyatlov_2013}. Here we shall not rely on this calculus explicitly but rather borrow some of its ideas and proceed as follows.  First in Proposition \ref{prop_localization_eigenfunction}  we establish a microlocalization of the null solutions $u$ of $P$ to a slightly larger region 
$X=\mathcal{O}(h^{\delta})$, using the standard $h^{\delta}$--calculus with $0<\delta<1/2$.  We refer  to \cite{Sjostrand_1992} for a similar microlocalization in a closely related context. Secondly, using  the sharp  G{\aa}rding inequality,  we get an a priori estimate for $P$, involving a microlocal cutoff,  regular on the scale $(h/\tilde h)^{1/2}$, see Proposition \ref{prop_Garding_my}. Using the a priori estimate and the microlocalization,  we obtain a uniform control in $L^2$ on 
\[
\text{Op}_h^w(q_0^N(X \tilde h^{1/2}/h^{1/2}))u,
\]
where $q_0$ is the quadratic approximation of $p_0$ and  $N$ large, see Proposition \ref{prop_test_Gaarding}.  The proof of Theorem \ref{thm_main} is concluded by a Sobolev embedding argument.

\section{Proof of Theorem \ref{thm_main}}

\subsection{A rough microlocalization of the ground states} To state our microlocalization result we have to introduce some notation.  Let $m\ge 1$ be a $C^\infty$ order function on $\R^{2n}$, i.e. there exist $C_0\ge 1$ and $N_0> 0$ such that 
\[
m(X)\le C_0\langle X-Y\rangle^{N_0} m(Y), \quad X,Y\in \R^{2n}.
\]
For $0\le \delta\le \frac{1}{2}$, we consider the following symbol class, 
\begin{align*}
S_{\delta}(m)=\{ a(X;h)\in C^\infty(\R^{2n};\C): \forall \alpha\in \N^{2n}, \exists C_\alpha>0,\forall h\in(0,1],\\
\forall X\in \R^{2n},|\p^\alpha_X a(X;h)|\le C_\alpha h^{-\delta|\alpha|}m(X)\}.
\end{align*}

We shall need the following composition formula for the Weyl quantization, see \cite{Dimassi_Sjostrand},  \cite{Zworski_book}, and \cite{Bony_Fujiie_Ramond_Zerzeri}.
If $a_1\in S_{\delta_1}(m_1)$ and $a_2\in S_{\delta_2}(m_2)$ with $0\le \delta_1,\delta_2\le 1/2$ and $\delta_1+\delta_2<1$, then 
\begin{equation}
\label{eq_composition_rule_Weyl}
\text{Op}_h^w(a_1)\text{Op}_h^w(a_2)=\text{Op}_h^w(a_1\# a_2),\quad a_1\# a_2\in S_{\max(\delta_1,\delta_2)}(m_1 m_2),
\end{equation}
and 
\[
(a_1\# a_2)(x,\xi)=e^{\frac{ih}{2} \sigma(D_x,D_\xi; D_y,D_\eta)}(a_1(x,\xi) a_2(y,\eta))|_{\substack{y=x\\ \eta=\xi}},
\]
where
\[
\sigma(D_x,D_\xi; D_y,D_\eta)=D_\xi\cdot D_y-D_x\cdot D_\eta.
\]

By Taylor's formula, applied to  $t\mapsto e^{\frac{iht}{2} \sigma(D_x,D_\xi; D_y,D_\eta)}$, for any $N\in \N$, we have 
\begin{equation}
\label{eq_taylor_comp_symb}
\begin{aligned}
(a_1\# a_2)(x,\xi)=\sum_{k=0}^N\frac{1}{k!}  \big(\frac{ih}{2} \sigma(D_x,D_\xi; D_y,D_\eta)\big)^k
(a_1(x,\xi) a_2(y,\eta))|_{\substack{y=x\\ \eta=\xi}}
+\frac{1}{N!}\\
\times \int_0^1 (1-t)^N e^{\frac{iht}{2} \sigma(D_x,D_\xi; D_y,D_\eta)}
\big(\frac{ih}{2} \sigma(D_x,D_\xi; D_y,D_\eta)\big)^{N+1}
(a_1(x,\xi) a_2(y,\eta))|_{\substack{
   y=x \\
   \eta=\xi}}
dt.
\end{aligned}
\end{equation}
It follows that
\begin{equation}
\label{eq_taylor_comp_symb_conclusion}
\begin{aligned}
(a_1\# a_2)(x,\xi)-\sum_{k=0}^N\frac{1}{k!}  \big(\frac{ih}{2} \sigma(D_x,D_\xi; D_y,D_\eta)\big)^k
(a_1(x,\xi) a_2(y,\eta))|_{\substack{y=x\\ \eta=\xi}}\\
\in h^{(N+1)(1-\delta_1-\delta_2)}S_{\max(\delta_1,\delta_2)}(m_1m_2).
\end{aligned}
\end{equation}

We shall also need the following formula from \cite[p. 45]{Martinez_book}, valid for $k=1,2,\dots$,
\begin{equation}
\label{eq_1_15}
\begin{aligned}
\sigma(D_x,D_\xi; D_y,D_\eta)^k&(a_1(x,\xi) a_2 (y,\eta))|_{\substack{y=x\\ \eta=\xi}}\\
&=\sum_{|\alpha|+|\beta|=k} (-1)^{|\alpha|} \frac{k!}{\alpha!\beta!} (\p_\xi^\alpha \p_x^\beta a_1(x,\xi)) (\p_x^\alpha \p_\xi^\beta a_2(x,\xi)).
\end{aligned}
\end{equation}

The main result of this subsection is as follows. 
\begin{prop}
\label{prop_localization_eigenfunction}
Assume that $u\in L^2(\R^n)$, $\|u\|_{L^2}=1$, is such that 
\begin{equation}
\label{eq_prop_new_loc_1}
(\emph{\text{Op}}_h^w(p_0)+h\emph{\text{Op}}_h^w(p_1))u=0 \quad \text{on}\quad \R^n, \quad n\ge 2.
\end{equation}
Then  there is $\psi\in C^\infty_0(\R^{2n},[0,1])$ such that for any $0<\delta<1/2$, there exists $h_0>0$ such that for all $h\in (0,h_0]$, we have
\begin{equation}
\label{eq_loc_eigen_lem}
u=\emph{\text{Op}}_h^w (\psi (X/h^\delta))u +Ru,
\end{equation}
where $R\in h^{M(1-2\delta)}\emph{\text{Op}}_h^w(S_{\delta}(\langle X\rangle^{-N}))$ for any $M, N\in \N$. 
\end{prop}

\begin{proof}
Let $\chi\in C^\infty_0(\R^{2n},[0,1])$ be such that $\chi(X)=1$ for $|X|\le 1$ and $\supp(\chi)\subset \{X\in \R^{2n}: |X|\le 2\}$.  
Since $p_0$ is not elliptic near zero,  to prove \eqref{eq_loc_eigen_lem} we consider the symbol 
\begin{equation}
\label{eq_1_6_0}
\tilde p(X;h)= p_0(X)+hp_1(X;h) +h^{2\delta}\chi(X/h^\delta),
\end{equation}
where $0<\delta<1/2$ is fixed, and construct a parametrix for the operator $\text{Op}_h^w(\tilde p)$. In doing so we shall proceed similarly to the proof of the sharp G{\aa}rding inequality in \cite{Dimassi_Sjostrand}.   

First let us show that there is $C>0$ such that 
\begin{equation}
\label{eq_est_re_p_0}
\Re p_0(X)\ge |X|^2/C,\quad X\in \R^{2n}.
\end{equation}
Indeed, when $|X|\le c_0$ with $c_0>0$ being a small but fixed constant, the estimate \eqref{eq_est_re_p_0}
follows from the quadratic approximation \eqref{eq_assump_quadratic_form} together with \eqref{eq_assumption_2_1_1}.  When $|X|\ge C_0$ with $C_0>0$ being a large but fixed constant, the estimate \eqref{eq_est_re_p_0} follows from \eqref{eq_assumption_4}. Finally, when $c_0 \le |X|\le C_0$, using  \eqref{eq_assumption_1} and the fact that $\Re p_0$ vanishes only at $X=0$, we conclude that $\Re p_0(X)\ge c>0$, and hence, \eqref{eq_est_re_p_0} follows.  

Now as a consequence of \eqref{eq_est_re_p_0}, we have 
\begin{equation}
\label{eq_1_6}
\Re p_0(X) +h^{2\delta}\chi(X/h^\delta)\ge \frac{h^{2\delta}}{C}\langle X\rangle^2, \quad X\in \R^{2n}.
\end{equation}
Indeed, when $|X|/h^\delta\ge 1$,  \eqref{eq_1_6} follows from \eqref{eq_est_re_p_0}, and when $|X|/h^\delta\le 1$, the estimate \eqref{eq_1_6} is a consequence of  \eqref{eq_assumption_1} and the fact that $\chi(X/h^\delta)=1$ in this region. 

Using \eqref{eq_assumption_p_1_der} and Taylor's formula, we see that
\begin{equation}
\label{eq_1_p_1_est}
|p_1(X;h)|\le C\langle X\rangle^2, \quad X\in \R^{2n},
\end{equation}
uniformly in $h\in (0,1]$, and therefore,  since $0<\delta<1/2$, there exists $h_{0}=h_0(\delta)>0$ such that  for  $0<h<h_0$ we have
\begin{equation}
\label{eq_1_7}
\Re \tilde p(X;h)\ge \frac{h^{2\delta}}{C}\langle X\rangle^2, \quad X\in \R^{2n}.
\end{equation}

We shall next estimate $\p^\alpha (1/\tilde p)$. To that end, we use Fa\`a di Bruno's  formula,
\begin{equation}
\label{eq_1_7_Faa_di_Bruno}
\p^{\alpha} f^{-1}=f^{-1} \sum_{k=1}^{|\alpha|}\sum_{\alpha=\beta^1+\dots+\beta^k,|\beta^j|\ge 1}C_{\beta^1,\dots,\beta^k}\prod_{j=1}^k(f^{-1}\p^{\beta^j}f),
\end{equation}
for appropriate constants $C_{\beta^1,\dots,\beta^k}$,
see \cite[p.94]{Zworski_book}. Using \eqref{eq_assumption_5} and \eqref{eq_assumption_p_1_der}, for $|\beta|\ge 2$, we get
\begin{equation}
\label{eq_1_8}
|\p^\beta\tilde p(X;h)|\le C_\beta h^{\delta(2-|\beta|)}, \quad X\in \R^{2n}. 
\end{equation}
This estimate together with \eqref{eq_1_7} implies that for $|\beta|\ge 2$,
\begin{equation}
\label{eq_1_8_2}
\bigg| \frac{\p^\beta \tilde p}{\tilde p}\bigg|\le C_\beta h^{-\delta|\beta|}\langle X\rangle^{-2},\quad X\in \R^{2n}.
\end{equation}

Let $|\beta|=1$. Here we need the following gradient estimate. Let $f:\R^n\to \R$ be $C^2$ with $f''\in L^\infty(\R^n)$, and $f\ge 0$, then 
\begin{equation}
\label{eq_1_9}
|\nabla f(x)|^2\le 2\|f''\|_{L^\infty(\R^n)} f(x),
\end{equation}
see \cite[Lemma 4.31]{Zworski_book}. We have therefore, 
\begin{equation}
\label{eq_1_10}
|\p^\beta (\Re \tilde p)|\le C(\Re \tilde p)^{1/2},  \quad |\beta|=1,
\end{equation}
with $C>0$ independent of $h$.  This together with \eqref{eq_1_7} implies that 
\begin{equation}
\label{eq_1_10_full_est}
\bigg| \frac{\p^\beta  \Re \tilde p}{\tilde p}\bigg|\le C| \tilde p|^{-1/2}\le Ch^{-\delta} \langle X\rangle^{-1},  \quad |\beta|=1,
\end{equation}
for all $0<h<1$ small enough.

Let us now estimate the gradient of $\Im \tilde p$. 
By \eqref{eq_assumption_5}, \eqref{eq_assump_quadratic_form} and \eqref{eq_est_re_p_0},  we get
\begin{equation}
\label{eq_est_im_part_p_1}
|\Im p_0(X)|\le C|X|^2\le C \Re p_0(X).
\end{equation}
We also have 
\begin{equation}
\label{eq_est_im_part_p_2}
 \Re p_0(X)\le C \Re \tilde p(X;h),
\end{equation}
for $h\in (0,1]$. Indeed, using \eqref{eq_1_p_1_est} and \eqref{eq_1_7}, we get
\begin{align*}
\Re p_0(X)&=\Re \tilde p(X;h)-h\Re p_1(X;h)-h^{2\delta}\chi(X/h^\delta)\\
&\le\Re \tilde p(X;h) +h|\Re p_1(X;h)|\le \Re \tilde p(X;h) +Ch\langle X\rangle^2\le C\Re \tilde p(X;h),
\end{align*}
showing \eqref{eq_est_im_part_p_2}. Thus, it follows from \eqref{eq_est_im_part_p_1} and \eqref{eq_est_im_part_p_2} that 
\[
 C \Re \tilde p(X;h) -\Im p_0(X)\ge 0,
\]
 and therefore, using \eqref{eq_1_9} and \eqref{eq_1_10}, we obtain that 
\begin{equation}
\label{eq_1_11}
\begin{aligned}
|\p^\beta \Im  p_0|&\le |\p^\beta(C\Re \tilde p -\Im  p_0)|+C|\p^\beta\Re \tilde p|\\
&\le C (C\Re \tilde p -\Im  p_0)^{1/2}+ C(\Re \tilde p)^{1/2}\le C(\Re\tilde p)^{1/2},\quad   |\beta|=1.
\end{aligned}
\end{equation}

Using that 
\[
\Im \tilde p=\Im p_0+h\Im p_1,
\]
\[
|\p^\beta p_1(X;h)|\le C\langle X\rangle,\quad |\beta|=1,
\]
uniformly in $h\in (0,1]$, and \eqref{eq_1_11}, \eqref{eq_1_7}, we  get
\begin{equation}
\label{eq_1_12}
\bigg| \frac{\p^\beta  \Im \tilde p}{\tilde p}\bigg|\le C| \tilde p|^{-1/2}+ \frac{C h \langle X\rangle}{|\tilde p|} \le Ch^{-\delta} \langle X\rangle^{-1},  \quad |\beta|=1,
\end{equation}
for  all $0<h<1$ small enough.

Combining \eqref{eq_1_8_2}, \eqref{eq_1_10_full_est} and \eqref{eq_1_12}, we write
\begin{equation}
\label{eq_1_13}
\bigg| \frac{\p^\beta  \tilde p}{\tilde p}\bigg|\le Ch^{-|\beta|\delta} \langle X\rangle^{-1}, \quad |\beta|\ge 1, \quad X\in \R^{2n}. 
\end{equation}
Letting $e(X;h)=1/\tilde p$, and using \eqref{eq_1_7_Faa_di_Bruno} together with \eqref{eq_1_7} and \eqref{eq_1_13}, we obtain that 
\begin{equation}
\label{eq_1_14}
|\p^\alpha e|\le C_\alpha h^{-2\delta-\delta|\alpha|}\langle X \rangle^{-2}, \quad |\alpha|\ge 0,
\end{equation}
i.e. $h^{2\delta}e\in S_{\delta}(\langle X\rangle^{-2})$.

Using  \eqref{eq_taylor_comp_symb} with $N=1$ and the fact that the Poisson bracket 
$\{e,\tilde p\}=0$, 
we get
\begin{equation}
\label{eq_1_15_0}
\begin{aligned}
(e&\#\tilde p)(x,\xi)=1\\
&+\frac{1}{4}\int_0^1 (1-t) e^{\frac{iht}{2} \sigma(D_x,D_\xi; D_y,D_\eta)} (ih \sigma(D_x,D_\xi; D_y,D_\eta))^2(e(x,\xi)\tilde p(y,\eta))|_{\substack{y=x\\ \eta=\xi}}dt.
\end{aligned}
\end{equation}
Next we would like to determine the symbol class of the integrand in \eqref{eq_1_15_0} uniformly in $t$.  
To that end, in view of \eqref{eq_1_15}, we first conclude from \eqref{eq_1_14} that  
\begin{equation}
\label{eq_1_16}
\p_\xi^\alpha \p_x^\beta e(x,\xi)\in h^{-4\delta} S_{\delta}(\langle X\rangle^{-2}),\quad |\alpha|+|\beta|=2,
\end{equation}
and from \eqref{eq_1_6_0} and \eqref{eq_assumption_5} that 
\begin{equation}
\label{eq_1_17}
\p_y^\alpha \p_\eta^\beta \tilde p(y,\eta)\in S_{\delta} (1), \quad |\alpha|+|\beta|=2.
\end{equation}
Thus, using \eqref{eq_1_15}, \eqref{eq_1_16} and \eqref{eq_1_17}, we get
\begin{equation}
\label{eq_1_18}
h^2 \sigma(D_x,D_\xi; D_y,D_\eta)^2(e(x,\xi)\tilde p(y,\eta))\in h^{2-4\delta}S_{\delta}(\langle X \rangle^{-2} ).
\end{equation}

Using the fact that 
\[
e^{\frac{iht}{2} \sigma(D_x,D_\xi; D_y,D_\eta)}: S_\delta (\langle X \rangle^{-2})\to S_\delta(\langle X \rangle^{-2}), 
\]
see \cite[Theorem 4.17]{Zworski_book},  and  \eqref{eq_1_18},  we obtain from   \eqref{eq_1_15_0}
 that 
\[
e\#\tilde p=1+h^{2-4\delta} r, \quad r\in S_{\delta}(\langle X \rangle^{-2}).
\]
Hence, 
\begin{equation}
\label{eq_1_19}
\text{Op}_h^w(e)\text{Op}_h^w(\tilde p)=1+h^{2-4\delta}\text{Op}_h^w(r),
\end{equation} 
where the operator $\text{Op}_h^w(r)=\mathcal{O}(1):L^2(\R^n)\to L^2(\R^n)$ is bounded for all $0<h$ small enough, see  \cite[Theorem 4.23]{Zworski_book}.  As $0<\delta<1/2$, we have 
\[
\|h^{2-4\delta}\text{Op}_h^w(r)\|_{L^2(\R^n)\to L^2(\R^n)}<1/2,
\]
for all $0<h$ small enough and therefore, the inverse $(1+h^{2-4\delta}\text{Op}_h^w(r))^{-1}$ exists as an operator $L^2(\R^n)\to L^2(\R^n)$. 

Next using that $1+h^{2-4\delta}r\in S_{\delta}(1)$ and Beals's theorem for $S_\delta(1)$, see \cite[p. 176 -- 177] {Zworski_book},
  we see that $(1+h^{2-4\delta}\text{Op}_h^w(r))^{-1}:=\text{Op}_h^w(r_1)$ is a pseudodifferential operator with $r_1\in S_{\delta}(1)$. 
 
It follows from \eqref{eq_1_19}  that for all $0<h$ small enough, we have
\[
\text{Op}_h^w(r_1)\text{Op}_h^w(e)\text{Op}_h^w(\tilde p)=1.
\]
Using the composition formula \eqref{eq_composition_rule_Weyl},  we see  that 
\[
\text{Op}_h^w(r_1)\text{Op}_h^w(e)h^{2\delta}=\text{Op}_h^w(r_2),  \quad r_2 \in S_\delta(\langle X\rangle^{-2}).
\]
This together with  \eqref{eq_1_6_0}, and the fact that $(\text{Op}_h^w(p_0)+h\text{Op}_h^w(p_1))u=0$ implies that 
\begin{equation}
\label{eq_1_20}
u=\text{Op}_h^w(r_2)\text{Op}_h^w(\chi(X/h^\delta))u.
\end{equation}
Let $\psi\in C_0^\infty(\R^{2n},[0,1])$ be such that $\psi=1$ near $\supp(\chi)$ and 
\[
\supp(\psi)\subset \{X\in \R^{2n}: |X|\le 3\}.
\]
Then it follows from \eqref{eq_1_20} that 
\[
u=\text{Op}_h^w(\psi(X/h^\delta))u +R u,
\]
where 
\[
R=(1- \text{Op}_h^w(\psi(X/h^\delta)))\text{Op}_h^w(r_2)\text{Op}_h^w(\chi(X/h^\delta)).
\]  
Here we notice that 
\begin{align*}
\chi(X/h^{\delta})\in S_\delta(\langle X\rangle^{-N}), \quad \forall N\in \N, \quad\text{and}\quad 1-\psi(X/h^{\delta})\in S_{\delta}(1).
\end{align*}
Since $\supp(1-\psi) \cap \supp(\chi)=\emptyset$,  it follows from \eqref{eq_taylor_comp_symb}  that 
\[
R\in h^{M(1-2\delta)}\text{Op}_h^w(S_{\delta}(\langle X\rangle^{-N})), 
\]
for any $N, M\in \N$. The proof is complete. 
\end{proof}

It follows from Proposition \ref{prop_localization_eigenfunction} that if $u\in L^2(\R^n)$ satisfies  \eqref{eq_prop_new_loc_1} then $u\in \mathcal{S}(\R^n)$, the Schwartz space.

\subsection{Applying G{\aa}rding's inequality} 

We shall need the following version of the sharp G{\aa}rding inequality, see \cite{Tataru_Fefferman-Phong} and \cite{Bony_J_M}. 
\begin{thm}
\label{thm_Gaarding}
Let $a(x,\xi;h)\in C^\infty(\R^{2n})$ be such that $a\ge 0$ on $\R^{2n}$ and $\p^\alpha a\in L^\infty(\R^{2n})$ for all $|\alpha|\ge 2$. Then there exist $C>0$, depending only on $\|\p^\alpha a\|_{L^\infty}$, $|\alpha|\ge 2$, and $h_0>0$ such that 
\[
(\emph{\text{Op}}^w_h(a)u,u)_{L^2(\R^n)}\ge -C h\|u\|^2_{L^2(\R^n)},
\]
for all $0<h\le h_0$ and $u\in \mathcal{S}(\R^n)$. 
\end{thm}

We shall now establish a suitable a priori estimate for the operator $P=\text{Op}_h^w(p_0)+h\text{Op}_h^w(p_1)$. To that end,  we let $0<\tilde h$ be sufficiently small but independent of $h$. We shall view $\tilde h$ as a second semiclassical parameter. In order to relate the $h$--Weyl quantization and $\tilde h$--Weyl quantization,  following \cite{Sjostrand_Zworski_2007}, we set 
\[
x=\sqrt{\varepsilon}\tilde x,\quad \xi=\sqrt{\varepsilon}\tilde \xi, \quad y=\sqrt{\varepsilon}\tilde y, \quad \varepsilon=h/\tilde h.  
\]
We obtain that
\[
(\text{Op}_h^w(a) u)(x)=\varepsilon^{-\frac{n}{4}} (\text{Op}_{\tilde h}^w(\tilde a) \tilde u)(\tilde x),
\]
where 
\begin{equation}
\label{eq_1_2_0_change}
\begin{aligned}
\tilde a (\tilde x, \tilde \xi)=a (\sqrt{\varepsilon} \tilde x,\sqrt{\varepsilon} \tilde \xi), \quad \tilde u(\tilde x)=\varepsilon^{\frac{n}{4}} u(\sqrt{\varepsilon} \tilde x). 
\end{aligned}
\end{equation}
Letting  
\begin{equation}
\label{eq_1_2_0_U}
U: u(x)\mapsto \tilde u(\tilde x)=\varepsilon^{\frac{n}{4}}u(\sqrt{\varepsilon}\tilde x),
\end{equation}
one can easily see that $U$ is unitary on $L^2(\R^n)$, and we have 
\begin{equation}
\label{eq_1_2}
\text{Op}_h^w(a)=U^{-1}  \text{Op}_{\tilde h}^w(\tilde a) U.
\end{equation}

We have the following consequence of Theorem \ref{thm_Gaarding}. 
\begin{prop} 
\label{prop_Garding_my}
 Let $\chi\in C^\infty_0(\R^{2n},[0,1])$ be such that $\chi(X)=1$ for $|X|\le 1$ and $\supp(\chi)\subset \{X\in \R^{2n}: |X|\le 2\}$. Then there exist  $\tilde C>0$ and $\tilde h_0>0$ such that 
\begin{equation}
\label{eq_prop_Garding_my}
\emph{\Re} ((\emph{\text{Op}}_h^w(p_0)+h\emph{\text{Op}}_h^w(p_1))u, u)_{L^2(\R^n)}+\varepsilon (\emph{\text{Op}}_h^w (\chi (X/\sqrt{\varepsilon} ))u, u)_{L^2(\R^n)}\ge \frac{\varepsilon}{\tilde C}\|u\|^2_{L^2(\R^n)}, 
\end{equation}
for  all $0<h\le \tilde h\le \tilde h_0$ and $u\in \mathcal{S}(\R^n)$.  Here $\varepsilon =h/\tilde h$.
\end{prop}

\begin{proof}
To establish \eqref{eq_prop_Garding_my},  using   \eqref{eq_1_2}, we  pass to the $\tilde h$--Weyl quantization and  get
\begin{equation}
\label{eq_1_4_0}
\text{Op}_h^w(p_0)+h\text{Op}_h^w(p_1) +\varepsilon\text{Op}_{ h}^w (\chi (X/\sqrt{\varepsilon} ))=\varepsilon U^{-1}\text{Op}_{\tilde h}^w(\tilde p) U,
\end{equation}
where 
\begin{equation}
\label{eq_1_4}
\tilde p(X;\varepsilon,\tilde h)= \frac{1}{\varepsilon} p_0(\sqrt{\varepsilon} X) +\tilde h p_1(\sqrt{\varepsilon} X; \varepsilon \tilde h) +\chi(X).
\end{equation}

Let us show that there is $C>0$ such that for $0<\tilde h$ small enough, 
\begin{equation}
\label{eq_1_3}
\Re \tilde p(X;\varepsilon,\tilde h)\ge 1/C, \quad X\in \R^{2n},
\end{equation}
uniformly in $\varepsilon$. Indeed, when $|X|\le 1$, the estimate \eqref{eq_1_3} follows from  \eqref{eq_assumption_1}, \eqref{eq_1_p_1_est}, and the fact that $\chi(X)=1$ here. When $|X|\ge 1$,    \eqref{eq_1_3} is implied by \eqref{eq_est_re_p_0} and \eqref{eq_1_p_1_est}. 

Using \eqref{eq_assumption_5} and \eqref{eq_assumption_p_1_der}, for $|\alpha|\ge 2$, we get
\begin{align*}
|\p^\alpha \tilde p(X;\varepsilon,\tilde h)|&\le \frac{(\sqrt{\varepsilon})^{|\alpha|}}{\varepsilon}|(\p^{\alpha}p_0)(\sqrt{\varepsilon}X)| +\tilde h(\sqrt{\varepsilon})^{|\alpha|}|(\p^\alpha p_1)(\sqrt{\varepsilon} X; \varepsilon \tilde h)|  +|\p^\alpha \chi(X)|\\
&\le C_\alpha, \quad X\in \R^{2n},
\end{align*}
uniformly in $\varepsilon\le 1$ and $\tilde h\in (0,1]$.  Applying Theorem \ref{thm_Gaarding} to $\Re \tilde p$ in the $\tilde h$--Weyl quantization, we obtain that there exist  $\tilde C>0$ and  $\tilde h_0>0$ such that  
\begin{equation}
\label{eq_1_5}
\Re (\text{Op}_{\tilde h}^w(\tilde p) u,u)_{L^2(\R^n)}\ge \frac{1}{\tilde C}\|u\|^{2}_{L^2(\R^n)},
\end{equation}
for all $0<h\le \tilde h\le \tilde h_0$  and $u\in \mathcal{S}(\R^n)$.

Using \eqref{eq_1_4_0}, \eqref{eq_1_5} and the fact that $U$ is unitary on $L^2(\R^n)$, we obtain that  
\begin{align*}
\Re ((\text{Op}_h^w(p_0)+h\text{Op}_h^w(p_1))u, u)_{L^2(\R^n)}+\varepsilon (\text{Op}_h^w (\chi (X/\sqrt{\varepsilon}))u, u)_{L^2(\R^n)}\\
=\varepsilon \Re (\text{Op}_{\tilde h}^w(\tilde p) Uu,Uu)_{L^2(\R^n)}\ge \frac{\varepsilon}{\tilde C}\|u\|^2_{L^2(\R^n)}, 
\end{align*} 
for all $0<h\le \tilde h\le h_0$  and $u\in \mathcal{S}(\R^n)$. This completes the proof.
\end{proof}

\subsection{Testing the a priori estimate}
In what follows we shall take $\tilde h>0$ sufficiently small but fixed, i.e. independent of $h$, so that Proposition \ref{prop_Garding_my}  is valid. The dependence on the parameter $\tilde h$ will
therefore not be indicated explicitly. 

The following result obtained by combining Proposition \ref{prop_localization_eigenfunction} and Proposition  \ref{prop_Garding_my}  is an essential step in the proof of Theorem \ref{thm_main}. 
\begin{prop}
\label{prop_test_Gaarding}
Assume that 
\[
(\emph{\text{Op}}_h^w(p_0)+h\emph{\text{Op}}_h^w(p_1))u=0\quad \text{on}\quad \R^n,\quad n\ge 2,
\]
$u\in L^2(\R^n)$, $\|u\|_{L^2}=1$. Set $q_0(X)=\frac{1}{2} p''_0(0)X\cdot X$. Then for every $N\in\N$, there exists $h_0>0$ such that for all $0<h\le h_0$, we have 
\begin{equation}
\label{eq_sub_3_1_0}
\| \emph{\text{Op}}_h^w (q_0^N (X/\sqrt{\varepsilon}))u\|_{L^2(\R^n)}\le \mathcal{O}_N(1), \quad \varepsilon=h/\tilde h.
\end{equation}
\end{prop}

\begin{proof} 
First using Proposition \ref{prop_localization_eigenfunction}, we see that $\text{Op}_h^w (q_0^N (X/\sqrt{\varepsilon}))u\in L^2(\R^n)$ for any $N\in \N$.  Thus, it follows from the a priori estimate \eqref{eq_prop_Garding_my} that there is $\tilde C>0$  such that 
\begin{equation}
\label{eq_sub_3_1}
\begin{aligned}
 \Re ((\text{Op}_h^w(p_0)&+h\text{Op}_h^w(p_1))\text{Op}_h^w (q_0^N (X/\sqrt{\varepsilon}))u, \text{Op}_h^w (q_0^N (X/\sqrt{\varepsilon}))u)_{L^2(\R^n)}\\
&+\varepsilon (\text{Op}_h^w (\chi (X/\sqrt{\varepsilon} ))\text{Op}_h^w (q_0^N (X/\sqrt{\varepsilon}))u, \text{Op}_h^w (q_0^N (X/\sqrt{\varepsilon}))u)_{L^2(\R^n)}\\
&\ge \frac{\varepsilon}{\tilde C}\|\text{Op}_h^w (q_0^N (X/\sqrt{\varepsilon}))u\|^2_{L^2(\R^n)}, 
\end{aligned}
\end{equation}
for all $0<h$ small enough and all $N\in \N$.   

Let us start by estimating the second term in the left hand side of \eqref{eq_sub_3_1}.  Using \eqref{eq_1_2_0_change}, \eqref{eq_1_2},  and the fact that $U$ is unitary, we have 
\begin{equation}
\label{eq_sub_3_2_-1}
\begin{aligned}
(\text{Op}_h^w & (\chi (X/\sqrt{\varepsilon} ))\text{Op}_h^w (q_0^N (X/\sqrt{\varepsilon}))u, \text{Op}_h^w (q_0^N (X/\sqrt{\varepsilon}))u)_{L^2(\R^n)}\\
&=
( \text{Op}_{\tilde h}^w (\overline{q}_0^N (X))  \text{Op}_{\tilde h}^w (\chi (X))\text{Op}_{\tilde h}^w (q_0^N (X))Uu, Uu)_{L^2(\R^n)}\le \mathcal{O}_N(1)\|u\|_{L^2(\R^n)}^2,
\end{aligned}
\end{equation}
for all $0<h$ small enough and all $N\in \N$. Here we have used the fact that $\chi$ has a compact support, and therefore, 
\[
\text{Op}_{\tilde h}^w (\overline{q}_0^N (X))  \text{Op}_{\tilde h}^w (\chi (X))\text{Op}_{\tilde h}^w (q_0^N (X))\in \text{Op}_{\tilde h}^w(S(1)),
\]
so that 
\[
\text{Op}_{\tilde h}^w (\overline{q}_0^N (X))  \text{Op}_{\tilde h}^w (\chi (X))\text{Op}_{\tilde h}^w (q_0^N (X))=\mathcal{O}_N(1): L^2(\R^n)\to L^2(\R^n)
\]
is bounded, see \cite[Theorem 4.23]{Zworski_book}

Let us consider the first term  in the left hand side of \eqref{eq_sub_3_1} and show that 
\begin{equation}
\label{eq_sub_3_2_0}
\begin{aligned}
 \Re (( \text{Op}_h^w(p_0)+h\text{Op}_h^w(p_1))\text{Op}_h^w (q_0^N (X/\sqrt{\varepsilon}))u, \text{Op}_h^w (q_0^N (X/\sqrt{\varepsilon}))u)_{L^2(\R^n)}\\
 \le  \mathcal{O}_{N}(h)\|u\|^2_{L^2(\R^n)}.
 \end{aligned}
 \end{equation}
Since $(\text{Op}_h^w(p_0)+h\text{Op}_h^w(p_1))u=0$, we get 
\begin{equation}
\label{eq_sub_3_2_0_two_commut}
\begin{aligned}
(( \text{Op}_h^w(p_0)&+h\text{Op}_h^w(p_1))\text{Op}_h^w (q_0^N (X/\sqrt{\varepsilon}))u, \text{Op}_h^w (q_0^N (X/\sqrt{\varepsilon}))u)_{L^2(\R^n)}\\
  = &( \text{Op}_h^w (\overline{q}_0^N (X/\sqrt{\varepsilon})) [ \text{Op}_h^w(p_0), \text{Op}_h^w (q_0^N (X/\sqrt{\varepsilon}))]u, u)_{L^2(\R^n)}\\
  &+ h( \text{Op}_h^w (\overline{q}_0^N (X/\sqrt{\varepsilon})) [ \text{Op}_h^w(p_1), \text{Op}_h^w (q_0^N (X/\sqrt{\varepsilon}))]u, u)_{L^2(\R^n)}.
\end{aligned}
\end{equation}

Thus, it suffices to show that  
\begin{equation}
\label{eq_sub_3_2_0_first_commut}
 ( \text{Op}_h^w (\overline{q}_0^N (X/\sqrt{\varepsilon})) [ \text{Op}_h^w(p_0), \text{Op}_h^w (q_0^N (X/\sqrt{\varepsilon}))]u, u)_{L^2(\R^n)} \le  \mathcal{O}_{N}(h)\|u\|^2_{L^2(\R^n)},
\end{equation}
and
\begin{equation}
\label{eq_sub_3_2_0_second_commut}
( \text{Op}_h^w (\overline{q}_0^N (X/\sqrt{\varepsilon})) [ \text{Op}_h^w(p_1), \text{Op}_h^w (q_0^N (X/\sqrt{\varepsilon}))]u, u)_{L^2(\R^n)} \le  \mathcal{O}_{N}(1)\|u\|^2_{L^2(\R^n)}.
\end{equation}

Let us start by establishing \eqref{eq_sub_3_2_0_first_commut}. To that end, 
since $q_0$ is quadratic, by the composition formula for the Weyl quantization \eqref{eq_taylor_comp_symb} we have 
\begin{equation}
\label{eq_sub_3_2}
[\text{Op}_h^w(q_0), \text{Op}_h^w(q_0^N) ]=\frac{h}{i}\text{Op}_h^w(\{q_0,q_0^N\})=0.
\end{equation}
Letting 
\[
r(X)=p_0(X)-q_0(X),
\]
and using \eqref{eq_sub_3_2}, we get
\begin{align*}
\text{Op}_h^w (\overline{q}_0^N (X/\sqrt{\varepsilon})) [ \text{Op}_h^w(p_0), &\text{Op}_h^w (q_0^N (X/\sqrt{\varepsilon}))]\\
&=\frac{1}{\varepsilon^{2N}} \text{Op}_h^w (\overline{q}_0^N (X)) [\text{Op}_h^w(r), \text{Op}_h^w (q_0^N (X))].
\end{align*}
We have
\begin{equation}
\label{eq_sub_3_3}
B:=\text{Op}_h^w (\overline{q}_0^N (X)) [\text{Op}_h^w(r), \text{Op}_h^w (q_0^N (X))]\in h\text{Op}_h^w(S_0(\langle X\rangle^{4N+2})),
\end{equation}
as $r\in  S_0(\langle X\rangle^2)$ in view of  \eqref{eq_assumption_5}, and $q_0^N\in S_{0}(\langle X\rangle^{2N})$.

By Proposition \ref{prop_localization_eigenfunction}, there exists $\psi\in C^\infty_0(\R^{2n},[0,1])$ such that for any $0<\delta<1/2$,  
we have for all $h>0$ small enough, 
\[
u=\text{Op}_h^w (\psi (X/h^\delta))u +Ru,
\]  
where  $R\in h^{M_1(1-2\delta)}\text{Op}_h^w(S_{\delta}(\langle X\rangle^{-M_2}))$ for any $M_1, M_2\in \N$.  Thus, 
\[
\varepsilon^{-2N} BR\in h^{-2N+ M_1(1-2\delta)} h\text{Op}_h^w(S_{\delta}(\langle X\rangle^{4N+2-M_2}))\subset h  \text{Op}_h^w(S_{\delta}(1)), 
\]
provided we choose $M_1$ and $M_2$ so large that 
\[
M_1\ge \frac{2N}{1-2\delta}, \quad M_2\ge 4N+2.
\]
Hence, the operator 
\[
\varepsilon^{-2N} BR=\mathcal{O}(h): L^2(\R^n)\to L^2(\R^n)
\]
is bounded for $0<h$ small enough.

Given $N\in \N$, let us choose $\delta$ so that 
\begin{equation}
\label{eq_our_choice_delta}
1/2>\delta\ge \frac{2N}{4N+1},
\end{equation}
and  show that the operator 
\begin{equation}
\label{eq_sub_3_4_0}
\varepsilon^{-2N} B\text{Op}_h^w (\psi (X/h^\delta))=\mathcal{O}(h): L^2(\R^n)\to L^2(\R^n)
\end{equation}
 is bounded for $0<h$ small enough.  To that end, first letting $B=\text{Op}_h^w(b)$,  using the composition formula \eqref{eq_taylor_comp_symb_conclusion} and the fact that $\psi (X/h^\delta)\in S_\delta(\langle X\rangle^{-L})$ for any $L\in \N$, and \eqref{eq_sub_3_3},  we write
\begin{equation}
\label{eq_sub_3_4}
\begin{aligned}
\varepsilon^{-2N}& b(x,\xi)\# \psi (x/h^\delta,\xi/h^\delta)\\
&=\varepsilon^{-2N} \sum_{j=0}^{K-1} \frac{1}{j!} \frac{(ih)^j}{2^j}\sigma(D_x,D_\xi;D_y,D_\eta)^j (b(x,\xi) \psi (y/h^\delta,\eta/h^\delta))|_{y=x,\eta=\xi}+\tilde r, 
\end{aligned}
\end{equation}
where 
\[
\tilde r\in \varepsilon^{-2N}h^{K(1-\delta)} S_\delta(\langle X\rangle^{4N+2-L}),
\]
for any $K\in \N$ and any $L\in \N$.  Choosing  
\[
L\ge 4N+2 \quad\text{and}\quad K\ge 4N+2,
\] 
we conclude that the operator 
\[
\text{Op}_h^w(\tilde r) =\mathcal{O}(h): L^2(\R^n)\to L^2(\R^n)
\]
 is bounded for all $0<h$ small enough.

To prove \eqref{eq_sub_3_4_0},  let us determine the symbol class for the first term in the right hand side of \eqref{eq_sub_3_4}, i.e. 
\begin{equation}
\label{eq_sub_3_4_b_tilde}
\tilde b(x,\xi)=\varepsilon^{-2N} \sum_{j=0}^{K-1} \frac{1}{j!} \frac{(ih)^j}{2^j}\sigma(D_x,D_\xi;D_y,D_\eta)^j (b(x,\xi) \psi (y/h^\delta,\eta/h^\delta))|_{y=x,\eta=\xi}.
\end{equation}
Using the composition formula \eqref{eq_taylor_comp_symb_conclusion},  \eqref{eq_1_15}, and the fact that $q_0$ is quadratic,  we get 
\begin{equation}
\label{eq_sub_3_5}
\begin{aligned}
b(x,\xi)=\sum_{l=0}^{2N} \frac{(ih)^l}{2^l} \sum_{k=1}^{2N} \frac{(ih)^k}{2^k}\sum_{|\alpha|+|\beta|=k} \frac{(-1)^{|\alpha|}}{\alpha!\beta!}\sum_{|\gamma|+|\delta|=l} \frac{(-1)^{|\gamma|}}{\gamma! \delta!} (\p_\xi^\gamma\p_x^\delta \overline{q}_0^N(x,\xi))\\
\p_x^\gamma\p_\xi^\delta \big[ (\p_\xi^\alpha\p_x^\beta r(x,\xi))(\p_x^\alpha\p_\xi^\beta q_0^N(x,\xi)) -(\p_\xi^\alpha\p_x^\beta q_0^N(x,\xi))(\p_x^\alpha\p_\xi^\beta r(x,\xi))\big].
\end{aligned}
\end{equation}

Hence, to estimate $\tilde b$, we see using \eqref{eq_sub_3_4_b_tilde},  \eqref{eq_sub_3_5},  and \eqref{eq_1_15} that we have to estimate the following terms,
\begin{equation}
\label{eq_sub_3_6}
\begin{aligned}
\varepsilon^{-2N} h^{j+l+k-\delta j}
&\p_\xi^{\mu}\p_x^\nu \bigg[ (\p_\xi^\gamma\p_x^\delta \overline{q}_0^N) \p_x^\gamma\p_\xi^\delta \big[ (\p_\xi^\alpha\p_x^\beta r)(\p_x^\alpha\p_\xi^\beta q_0^N) -(\p_\xi^\alpha\p_x^\beta q_0^N)(\p_x^\alpha\p_\xi^\beta r)\big]   \bigg] \\
&(\p_x^{\mu}\p_\xi^\nu \psi)(X/h^\delta), 
\end{aligned}
\end{equation}
where 
\begin{align*}
j=0,\dots, K-1, l=0,\dots, 2N, k=1,\dots, 2N,\\
|\alpha|+|\beta|=k, |\gamma|+|\delta|=l,|\mu|+|\nu|=j.
\end{align*}
It follows from \eqref{eq_sub_3_6} that it is enough to estimate
\begin{equation}
\label{eq_sub_3_7}
\varepsilon^{-2N} h^{j+l+k-\delta j} \p_X^{\mu} \bigg[ (\p_X^\gamma \overline{q}_0^N) \p_X^\gamma
 \big[ (\p_X^\alpha  r)(\p_X^\alpha q_0^N)\big]  \bigg], 
 \end{equation}
 on $\supp(\psi(X/h^\delta))$, i.e. when $|X|\le 3h^\delta$,
with 
\[
|\alpha|=k, |\gamma|=l, |\mu|=j.
\]
Using Leibniz's rule twice, we rewrite \eqref{eq_sub_3_7} as follows,
\begin{equation}
\label{eq_sub_3_8}
\varepsilon^{-2N} h^{j+l+k-\delta j} \sum_{\mu^1+\mu^2=\mu} C_{\mu^{1},\mu^2} (\p_X^{\mu^1+\gamma} \overline{q}_0^N)\bigg(\sum_{\gamma^1+\gamma^2=\mu^2+\gamma} C_{\gamma^1,\gamma^2} (\p_X^{\gamma^1+ \alpha}  r)(\p_X^{\gamma^2+ \alpha} q_0^N)\bigg).
\end{equation}

As $|\alpha|=k\ge 1$, we know that  $|\gamma^1|+|\alpha|\ge 1$. 
Consider first the case $|\gamma^1|+|\alpha|= 1$. In this case 
\[
|\p_X^{\gamma^1+ \alpha}  r|\le \mathcal{O}(|X|^2),
\] 
since 
\[
r(X)=\mathcal{O}(|X|^3)\quad \text{near}\quad 0.
\]
Therefore, using the fact that 
\[
|\p_X^\beta q_0^N|\le \begin{cases} \mathcal{O}(|X|^{2N-|\beta|}), & |\beta|\le 2N,\\
0, &  |\beta|> 2N,
\end{cases}
\]
we estimate the absolute value of \eqref{eq_sub_3_8} in the case $|\gamma^1|+|\alpha|= 1$  by
\begin{equation}
\label{eq_sub_3_9}
\begin{aligned}
&\le \varepsilon^{-2N} h^{j+l+k-\delta j}   \mathcal{O} (|X|^{4N-j-2l-k+|\gamma^1|})\mathcal{O}(|X|^2)\\
&\le \varepsilon^{-2N}  \mathcal{O}(h h^{\delta(4N+1)}h^{(1-2\delta)(j+l+k-1)})
\le \varepsilon^{-2N}  \mathcal{O}(h h^{\delta(4N+1)})\le \mathcal{O}(h).
\end{aligned}
\end{equation}
Here we have used that $4N -j-2l-k+|\gamma^1|\ge 0$ and  $1/2> \delta\ge \frac{2N}{4N+1}$.  Similarly,  using that  
\[
|\p_X^{\gamma^1+ \alpha}  r|\le \mathcal{O}(|X|)\quad \text{when}\quad |\gamma^1|+|\alpha|= 2, 
\]
and 
\[
|\p_X^{\gamma^1+ \alpha}  r|\le \mathcal{O}(1) \quad \text{when}\quad |\gamma^1|+|\alpha|\ge 3, 
\]
we obtain the estimate \eqref{eq_sub_3_9} also in the case when $|\gamma^1|+|\alpha|\ge 2$. 
Hence, we get
\[
|\tilde b(x,\xi)|\le  \mathcal{O}(h). 
\]

To estimate the derivatives $\p_X^{\rho}\tilde b(X)$, $|\rho|\ge 1$, arguing as above and using  Leibniz's rule one more time, we conclude that we have to estimate 
 \begin{equation}
 \label{eq_deriv_tilde_b}
 \varepsilon^{-2N} h^{j+l+k-\delta j-|\rho_2|\delta } \p_X^{\rho_1+\mu} \bigg[ (\p_X^\gamma \overline{q}_0^N) \p_X^\gamma
 \big[ (\p_X^\alpha  r)(\p_X^\alpha q_0^N)\big]  \bigg], 
 \end{equation}
 on $\supp(\psi(X/h^\delta))$,  with 
\[
|\rho|=|\rho_1|+|\rho_2|, 
|\alpha|=k, |\gamma|=l, |\mu|=j.
\]
Similarly to \eqref{eq_sub_3_8}, we write \eqref{eq_deriv_tilde_b} as follows,
\begin{align*}
\varepsilon^{-2N} h^{j+l+k-\delta j-|\rho_2|\delta } &\sum_{\mu^1+\mu^2=\rho_1+\mu} C_{\mu^{1},\mu^2} (\p_X^{\mu^1+\gamma} \overline{q}_0^N)\\
&\bigg(\sum_{\gamma^1+\gamma^2=\mu^2+\gamma} C_{\gamma^1,\gamma^2} (\p_X^{\gamma^1+ \alpha}  r)(\p_X^{\gamma^2+ \alpha} q_0^N)\bigg).
\end{align*}

Therefore, using that $4N-|\rho_1|- j-2l-k+|\gamma^1|\ge 0$, we get 
\begin{align*}
|\p_X^{\rho} \tilde b(x,\xi)|&\le \varepsilon^{-2N} h^{j+l+k-\delta j-|\rho_2|\delta}|\p_X^{\gamma^1+ \alpha}  r| \mathcal{O} (|X|^{4N-|\rho_1|- j-2l-k+|\gamma^1|})\\
&\le h^{-\delta|\rho|}  \mathcal{O}(h^{-2N}h h^{\delta(4N+1)}h^{(1-2\delta)(j+l+k-1)}) \le h^{-\delta|\rho|} \mathcal{O}(h),
\end{align*}
since $1/2>\delta\ge \frac{2N}{4N+1}$. Hence, 
\[
\tilde b\in hS_{\delta}(1),
\]
and thus,  \eqref{eq_sub_3_4_0} and \eqref{eq_sub_3_2_0_first_commut} follow. 

Let us now show \eqref{eq_sub_3_2_0_second_commut}. To that end, we write
 \[
  \text{Op}_h^w (\overline{q}_0^N (X/\sqrt{\varepsilon})) [ \text{Op}_h^w(p_1), \text{Op}_h^w (q_0^N (X/\sqrt{\varepsilon}))]=\frac{1}{\varepsilon^{2N}}W, 
 \]
where 
\[
W:= \text{Op}_h^w (\overline{q}_0^N (X)) [ \text{Op}_h^w(p_1), \text{Op}_h^w (q_0^N (X))]\in h\text{Op}_h^w(S_0(\langle X\rangle^{4N+2})),
\]
as $p_1\in S_0(\langle X\rangle^{2})$. Arguing as above, we see that it suffices to verify that 
\begin{equation}
\label{eq_sub_3_2_0_second_commut_2}
\text{Op}_h^w(\tilde w):=\varepsilon^{-2N}W\text{Op}_h^w(\psi(X/h^\delta))=\mathcal{O}(1):L^2(\R^n)\to L^2(\R^n),
\end{equation}
for $0<h$ small enough. Similarly to \eqref{eq_sub_3_8}, we observe that to bound $\tilde w$,  we have to estimate the expression
\[
\varepsilon^{-2N} h^{j+l+k-\delta j} \sum_{\mu^1+\mu^2=\mu} C_{\mu^{1},\mu^2} (\p_X^{\mu^1+\gamma} \overline{q}_0^N)\bigg(\sum_{\gamma^1+\gamma^2=\mu^2+\gamma} C_{\gamma^1,\gamma^2} (\p_X^{\gamma^1+ \alpha}  p_1)(\p_X^{\gamma^2+ \alpha} q_0^N)\bigg)
\]
 where  
\[
|\alpha|=k\ge 1, |\gamma|=l\ge 0, |\mu|=j\ge 0,
\]
in the region  $|X|\le 3h^\delta$. Using that $|\p^\beta p_1(X;h)|\le \mathcal{O}(1)$ for all $\beta$ and $|\gamma^2|\le |\mu^2|+|\gamma|$, we see that the expression above can be bounded by
\begin{align*}
\le C\varepsilon^{-2N} h^{j+l+k-\delta j} h^{\delta(4N-|\mu^1|-|\gamma|-|\gamma^2|-|\alpha|)}&\le 
C\varepsilon^{-2N}h^{4N\delta} h^{(j+l)(1-2\delta)}h^{k(1-\delta)}\\
&\le C\varepsilon^{-2N}h^{4N\delta+1-\delta}\le C,
\end{align*}
provided that $\delta\ge \frac{2N-1}{4N-1}$.  The latter condition is implied by \eqref{eq_our_choice_delta}, and therefore, we conclude that $|\tilde w|\le \mathcal{O}(1)$.  The derivatives of $\tilde w$ are estimated as above, and we get $\tilde w\in S_{\delta}(1)$. This shows \eqref{eq_sub_3_2_0_second_commut_2} and hence, \eqref{eq_sub_3_2_0_second_commut}.
The proof is complete. 
\end{proof}

\subsection{Concluding the proof of Theorem \ref{thm_main}} 
Let $N\in\N$ be fixed. Then by Proposition \ref{prop_test_Gaarding} and scaling \eqref{eq_1_2}, we have
\begin{equation}
\label{eq_3_10}
\| \text{Op}_{\tilde h}^w (q_0^N (X))Uu\|_{L^2(\R^n)}\le \mathcal{O}(1),
\end{equation}
for all $0<h$ small enough.  Now it is convenient to make an additional scaling to pass to the case $\tilde h=1$. By \eqref{eq_1_2} and the homogeneity of $q_0^N$, we have
\[
\text{Op}_{\tilde h}^w(q_0^N)=\tilde h^N V^{-1} \text{Op}_1^w (q_0^N)V, 
\]
where
\[
(Vu)(\tilde x)=(\tilde h)^{\frac{n}{4}}u(\sqrt{\tilde h}\tilde x).
\] 
Hence, in the remainder of the proof we may assume that $\tilde h=1$.

We have $q_0^N(X)\in S^{2N}_{X}(\R^{2n})$. Here 
\[
S^{m}_{X}(\R^{2n})=\{a(X)\in C^\infty(\R^{2n};\C): \forall a\in \N^{2n},\exists C_\alpha>0, |\p^\alpha a(X)| \le C_\alpha \langle X\rangle^{m-|\alpha|}\},
\]
see \cite[Section 23.1]{Shubin_book}.
Using the fact that $\Re q_0(X)$ is a positive definite quadratic form, we get
\[
|q_0^N(X)|\ge (\Re q_0(X))^N\ge |X|^{2N}/C, \quad X\ne 0.
\]
It follows from \cite[Theorem 25.1]{Shubin_book} that there is $b \in S^{-2N}_{X}(\R^{2n})$ such that 
\begin{equation}
\label{eq_3_11}
\text{Op}_1^w(b)\text{Op}_1^w(q_0^N)-I=R,
\end{equation}
where the operator $R$ has a kernel in the Schwartz space $\mathcal{S}(\R^{2n})$, and therefore,
\begin{equation}
\label{eq_3_12}
R: \mathcal{S}'(\R^{n})\to \mathcal{S}(\R^{n}).
\end{equation}
Here $ \mathcal{S}'(\R^{n})$ is the space of tempered distributions.

Let $s\in \R$ and let
\[
\mathcal{H}^s(\R^n)=\{u\in \mathcal{S}'(\R^n): \text{Op}_1^w((1+|x|^2+|\xi|^2)^{s/2})u\in L^2(\R^n)\}.
\]
We know that 
\begin{equation}
\label{eq_3_13}
\text{Op}_1^w(b): L^2(\R^n)\to  \mathcal{H}^{2N}(\R^n)
\end{equation}
is bounded, see  \cite[Theorem 25.2]{Shubin_book}.  It follows from \eqref{eq_3_11},  \eqref{eq_3_10}, \eqref{eq_3_12} and \eqref{eq_3_13} that 
\begin{equation}
\label{eq_3_14}
\begin{aligned}
\|Uu\|_{\mathcal{H}^{2N}(\R^n)}\le \|\text{Op}_1^w(b)\text{Op}_1^w(q_0^N) Uu\|_{\mathcal{H}^{2N}(\R^n)} +\|RUu\|_{\mathcal{H}^{2N}(\R^n)}\le \mathcal{O}(1),
\end{aligned}
\end{equation}
for all $0<h$ small enough.

Choosing $N>n/4$ and using the fact that $\mathcal{H}^{2N}(\R^n)\subset H^{2N}(\R^n)$, the standard Sobolev space, together with the Sobolev embedding $H^{2N}(\R^n)\subset L^\infty(\R^n)$, we get  
\[
\|Uu\|_{L^\infty(\R^n)}\le \mathcal{O}(1).
\]
Hence, recalling \eqref{eq_1_2_0_U}, we obtain that  
\[
\|u\|_{L^\infty(\R^n)}\le \mathcal{O}(1) h^{-n/4}. 
\]
This completes the proof of Theorem \ref{thm_main}.

\section*{Acknowledgements} 
We are very grateful to Maciej Zworski for helpful discussions. 
 The research of K.K. is partially supported by the National Science Foundation (DMS 1500703). The research of G.U. is partially supported by the National Science Foundation.

\end{document}